\documentclass[11pt,reqno]{amsart}
\usepackage[a4paper, hmargin={2.7cm,2.7cm},vmargin={3.3cm,3.3cm}]{geometry}
\usepackage[english]{babel}
\usepackage{color}
\usepackage[noadjust]{cite}
\usepackage{amsmath}
\usepackage{amssymb}
\usepackage{amsfonts}
\usepackage{amsthm}
\usepackage{enumerate}
\usepackage{amsthm}

\usepackage{etoolbox}

\numberwithin{equation}{section}

\makeatletter
\@namedef{subjclassname@2020}{\textup{2020} Mathematics Subject Classification}
\makeatother

\makeatletter
\patchcmd{\ttlh@hang}{\parindent\z@}{\parindent\z@\leavevmode}{}{}
\patchcmd{\ttlh@hang}{\noindent}{}{}{}
\makeatother

\newcommand\numberthis{\addtocounter{equation}{1}\tag{\theequation}}

\theoremstyle{plain}
\newtheorem{theorem}{Theorem}[section]
\newtheorem{lemma}[theorem]{Lemma}
\newtheorem{proposition}[theorem]{Proposition}
\newtheorem{corollary}[theorem]{Corollary}

\theoremstyle{definition}

\theoremstyle{remark}
\newtheorem{remark}[theorem]{Remark}

\DeclareMathOperator{\Aut}{Aut}

\DeclareMathOperator{\ind}{ind}

\title{Integrability properties of quasi-regular representations of $NA$ groups}

\author{Jordy Timo van Velthoven}

\address{Delft University of Technology,
Mekelweg 4, Building 36,
2628 CD Delft, The Netherlands.}
\email{j.t.vanvelthoven@tudelft.nl}



\begin{document}

\begin{abstract}
Let $G = N \rtimes A$,  where $N$ is a graded Lie group and $A = \mathbb{R}^+$ acts on $N$ via homogeneous dilations. The quasi-regular representation $\pi = \mathrm{ind}_A^G (1)$ of $G$ can be realised to act on $L^2 (N)$. It is shown that for a class of analysing vectors the associated wavelet transform defines an isometry from $L^2 (N)$ into $L^2 (G)$ and that the integral kernel of the corresponding orthogonal projector has polynomial off-diagonal decay. The obtained reproducing formula is instrumental for obtaining decomposition theorems for function spaces on nilpotent groups.
\end{abstract}

\maketitle

\section{Introduction}
Let $N$ be a connected, simply connected nilpotent Lie group and let $A = \mathbb{R}^+$ act on $N$ via automorphic dilations. The semi-direct product
$G = N \rtimes A$ acts unitarily on $L^2 (N)$ via the quasi-regular representation
$\pi = \ind_A^G (1)$ of $G$. For $g \in L^2 (N)$, the associated wavelet transform $V_{g} : L^2 (N) \to L^{\infty} (G)$ is defined as
\[
V_{g} f (x,t) = \langle f, \pi(x,t) g \rangle, \quad (x,t) \in G.
\]
A vector $g \in L^2 (N)$ is said to be \emph{admissible} if $V_{g}$ is an isometry from $L^2 (N)$ into $L^2 (G)$.

Given an admissible vector $g \in L^2 (N)$, the orthogonal projector $P$ from $L^2 (G)$ onto the closed subspace $V_{g} (L^2 (N)) \subset L^2 (G)$ is given by right convolution
$P(F) = F \ast V_{g} g$.
In particular, an element $F \in V_{g} (L^2 (N))$, i.e., $F = V_g f$ for some $f \in L^2 (N)$, satisfies the reproducing formula
\begin{align} \label{eq:repro}
V_{g} f = V_{g} f \ast V_{g} g.
\end{align}
The existence of admissible vectors for irreducible, square-integrable representations $\pi$ is automatic by the orthogonality relations \cite{duflo1976regular}, but a non-trivial problem for reducible representations. For $N = \mathbb{R}^d$ and general dilation groups $A \leq \mathrm{GL}(d, \mathbb{R})$, the admissibility of quasi-regular representations is well-studied, see, e.g. \cite{fuehr2010generalized, laugesen2002characterization,  bruna2015characterizing} and the references therein. For non-commutative groups $N$, the admissibility problem is considered in, e.g.
\cite{currey2012admissibility, currey2007admissibility, oussa2013admissibility, fuehr2005abstract}.

This note is concerned with admissible vectors that are also integrable:
A vector $g \in L^2 (N)$ is said to be \emph{integrable} if $\Delta_G^{-1/2} V_g g \in L^1 (G)$, where $\Delta_G : G \to \mathbb{R}^+$ denotes the modular function on $G$. The significance of integrably admissible vectors is that $F := \Delta_G^{-1/2} V_g g$ forms a \emph{projection} in $L^1 (G)$ by \eqref{eq:repro}, that is, $F  = F \ast F = F^*$, with $F^* := \Delta_G^{-1} \overline{F ( \cdot ^{-1})}$.

The construction of projections in $L^1 (G)$ arising from matrix coefficients is an ongoing research topic,
 and such projections provide (if they exist) a powerful tool for studying problems in non-commutative harmonic analysis.
 Among others, they play a vital role in the theory of atomic decompositions in Banach spaces \cite{feichtinger1989banach, grochenig1991describing}.

For the affine group $G = \mathbb{R} \rtimes \mathbb{R}^+$, the construction of projections in $L^1 (G)$ goes back to \cite{eymard1979transformation}. The papers \cite{kaniuth1996minimal, grochenig1992compact, currey2016integrable} consider groups $G = \mathbb{R}^d \rtimes A$ and provide criteria for the explicit construction of projections in $L^1 (G)$ based on the dual action of $A$ on  $\mathbb{R}^d$;
see also \cite{fuehr2015coorbit, fuehr2019coorbit}. The techniques of \cite{kaniuth1996minimal, grochenig1992compact} were used in \cite{schulz1999extensions} for the Heisenberg group $N = \mathbb{H}_1$ acted upon by automorphic dilations. For a stratified group $N$ with canonical dilations, the existence of smooth admissible vectors was investigated in \cite{geller2006continuous}, although not linked to integrability.

The main concern of this note is the integrability of $\pi = \ind^{N \rtimes A}_A$ when $N$ is a (possibly, non-stratified) graded Lie group. The main result obtained is the following:

\begin{theorem} \label{thm:intro}
Let $G = N \rtimes A$, where $N$ is a graded Lie group and $A = \mathbb{R}^+$
acts on $N$ via automorphic dilations. The quasi-regular representation $\pi = \ind^G_A (1)$ admits integrably admissible vectors, i.e., there exist vectors $g \in L^2 (N)$ satisfying $\Delta_G^{-1/2} V_{g} g \in L^1 (G)$ and
\[
\int_G | \langle f, \pi(x,t) g \rangle |^2 \; d\mu_G (x,t) = \| f \|_{2}^2, \quad \quad \text{for all} \;\; f \in L^2 (N).
\]
The integrably admissible vector $g$ can be chosen to be Schwartz with all moments vanishing, in which case $V_{g} g \in L^1_w (G)$ for any polynomially bounded weight $w : G \to [1,\infty)$.
\end{theorem}

Admissible vectors that are Schwartz with all vanishing moments are known to exist already for stratified Lie groups \cite[Corollary 1]{geller2006continuous}. Theorem \ref{thm:intro} provides a modest extension of this result to general graded Lie groups, and complements it with integrability properties of the associated matrix coefficients.
More explicit (point-wise) localisation estimates for the matrix coefficients on homogeneous groups are also obtained; see Section \ref{sec:matrixcoefficient} below for details.

The proof method for Theorem \ref{thm:intro} resembles the construction of Littlewood-Paley functions and Calder\'on-type reproducing formulae. Most techniques can already be found in some antecedent form in \cite{folland1982hardy} as pointed out throughout the text.
Particular use is made of the (non-stratified) Taylor inequality and Hulanicki's theorem for Rockland operators. The use of a Rockland operator instead of a sub-Laplacian is essential for the proof method as the latter are no longer always homogeneous for non-stratified groups. The exploitation of homogeneity is the reason that the strategy fails for non-graded homogeneous groups (see Remark \ref{rem:nonhomogeneous}).

The motivation for Theorem \ref{thm:intro} stems from the study of function spaces, and is twofold:

(i) The question whether there exist vectors yielding a reproducing kernel with suitable off-diagonal decay on  homogeneous groups was posed in \cite[Remark 6.6(a)]{grochenig1991describing}, where it was mentioned that this is a representation-theoretic problem rather than one of function spaces.
The use of such vectors for function space theory, however, is due to the fact that the techniques \cite{grochenig1991describing} yield frames and atomic decompositions for Besov-Triebel-Lizorkin spaces.
The same holds true for the recent sampling theorems in \cite{romero2021dual}. The admissible vectors provided by Theorem \ref{thm:intro} satisfy the integrability
conditions assumed in \cite{grochenig1991describing, romero2021dual} (see Section \ref{sec:analysing}), and Theorem \ref{thm:intro} solves the problem mentioned in \cite[Remark 6.6(a)]{grochenig1991describing} for graded Lie groups.

(ii) The differentiability properties of functions in terms of Banach spaces are well-studied on stratified Lie groups for several classes of spaces, including
Lipschitz spaces \cite{folland1979lipschitz, krantz1982lipschitz}, Sobolev spaces \cite{folland1975subelliptic, saka1979besov}, Besov spaces \cite{christensen2012coorbit, fuehr2012homogeneous, saka1979besov} and Triebel-Lizorkin spaces \cite{folland1982hardy, hu2013homogeneous}. More recently, there has been
an interest in such spaces on possibly non-stratified graded Lie groups, see, e.g. \cite{bahouri2012refined, fischer2017sobolev, cardona2017multipliers, bruno2021homogeneous}.
This was a motivation to obtain Theorem \ref{thm:intro} for graded groups, as it allows to apply the techniques \cite{grochenig1991describing, romero2021dual} discussed in (i) to these new classes of spaces. Moreover, even for stratified groups, the integrability properties provided by Theorem \ref{thm:intro} allow to apply the techniques \cite{romero2021dual} and bridge a gap between what has been established on the locality of the sampling expansions for stratified groups in \cite{christensen2012coorbit, fuehr2012homogeneous, grochenig1991describing, geller2006continuous} and for the classical setting $N = \mathbb{R}^d$ in \cite{frazier1991littlewood, gilbert2002smooth}; see \cite{gilbert2002smooth,romero2021dual} for more details on the discrepancy between \cite{grochenig1991describing} and  \cite{frazier1991littlewood, gilbert2002smooth, romero2021dual}.

The details on the applications of Theorem \ref{thm:intro} to various functional spaces are beyond the scope of the present paper, and will be deferred to subsequent work.

\subsection*{Notation} The open and closed positive half-lines in $\mathbb{R}$ are denoted by $\mathbb{R}^+ = (0,\infty)$ and $\mathbb{R}_0^+ = [0,\infty)$, respectively. For functions $f_1, f_2 : X \to \mathbb{R}_0^+$, it is written $f_1 \lesssim f_2$ if there exists a constant $C > 0$ such that $f_1 (x) \leq C f_2(x)$ for all $x \in X$.
The space of smooth functions on a Lie group $G$ is denoted by $C^{\infty} (G)$
and the space of test functions by $C_c^{\infty} (G)$.

\section{Preliminaries on homogeneous Lie groups} \label{sec:prelim}
This section provides background on homogeneous groups. Standard references for the theory are the books  \cite{folland1982hardy,fischer2016quantization}.
\subsection{Dilations}
Let $\mathfrak{n}$ be a real $d$-dimensional Lie algebra.
A \emph{family of dilations} on $\mathfrak{n}$ is a one-parameter family $\{ D_t \}_{t > 0}$
of automorphisms $D_t : \mathfrak{n} \to \mathfrak{n}$ of the form $D_t := \exp (A \ln t)$,
where $A : \mathfrak{n} \to \mathfrak{n}$ is a diagonalisable linear map with positive eigenvalues $v_1, ..., v_d$. If a Lie algebra $\mathfrak{n}$ is endowed with a family of dilations, then it is nilpotent.

A \emph{homogeneous group} is a connected, simply connected nilpotent Lie group $N$ whose Lie algebra $\mathfrak{n}$ admits a family of dilations. The number $Q := v_1 + ... + v_d$ is the \emph{homogeneous dimension} of $N$.
The exponential map $\exp_N : \mathfrak{n} \to N$ is a diffeomorphism, providing a global coordinate system on $N$. Dilations $\{D_t \}_{t > 0}$ can be transported to a one-parameter group of automorphisms of $N$, which will be denoted by $\{ \delta_t \}_{t > 0}$. The associated action of $t \in \mathbb{R}^+$ on $x \in N$ will often simply be written as $t x = \delta_t (x)$.

A \emph{graded group} is a connected, simply connected nilpotent Lie group $N$ whose Lie algebra $\mathfrak{n}$ admits an  $\mathbb{N}$-gradation $\mathfrak{n} = \bigoplus_{j = 1}^{\infty} \mathfrak{n}_j$,
where $\mathfrak{n}_j$, $j = 1, 2, ...,$ are vector subspaces of $\mathfrak{n}$, almost all equal to $\{0\}$, and satisfying $[\mathfrak{n}_j, \mathfrak{n}_{j'}] \subset \mathfrak{n}_{j+j'}$ for $j,j' \in \mathbb{N}$. If, in addition, $\mathfrak{n}_1$ generates $\mathfrak{n}$, the group $N$ is \emph{stratified}. Canonical dilations $D_t : \mathfrak{n} \to \mathfrak{n}$, $t > 0$, can be defined through a gradation as $D_t (X) = t^j X$ for $X \in \mathfrak{n}_j$, $j \in \mathbb{N}$.

Henceforth, a homogeneous group $N$ will be fixed with dilations $D_t := \exp (A \ln t)$. Haar measure will be denoted by $\mu_N$. The eigenvalues
$v_1, ..., v_d$ of $A$ will be listed in increasing order and it will be assumed (without loss of generality) that $v_1 \geq 1$. In addition, a basis $X_1, ..., X_d$ of $\mathfrak{n}$ such that $A X_j = v_j X_j$ for $j = 1, ..., d$ will be fixed throughout.

 \subsection{Homogeneity}
 A function $f : N  \to \mathbb{C}$ is called \emph{$\nu$-homogeneous} ($\nu \in \mathbb{C}$) if $f \circ \delta_t = t^{\nu} f$ for $t > 0$. For all measurable functions $f_1, f_2 : N \to \mathbb{C}$,
 \[
 \int_N f_1(x) (f_2 \circ \delta_t) (x) \; d\mu_N (x) = t^{-Q} \int_N (f_1 \circ \delta_{1/t} )(x) f_2(x) \; d\mu_N (x)
 \]
 provided the integral is convergent. The map $f \mapsto f \circ \delta_t$ is naturally extended to distributions.

 A linear operator $T : C_c^{\infty} (N) \to (C_c^{\infty} (N))'$ is said to be homogeneous of degree $\nu \in \mathbb{C}$ if $T (f \circ \delta_t) = t^{\nu} (Tf) \circ \delta_t$ for all $f \in C_{c}^{\infty} (N)$ and $t > 0$.

 A \emph{homogeneous quasi-norm} on $N$ is a continuous function $|\cdot| : N \to [0,\infty)$ that is
 symmetric, $1$-homogeneous and definite.
 If $|\cdot|$ is a homogeneous quasi-norm on $N$, there is a constant $C > 0$ such that
 $ |x y | \leq C(|x| + |y|) $
 for all $x, y \in N$.

\subsection{Derivatives and polynomials}
A basis element $X_j \in \mathfrak{n}$ acts as a left-invariant vector field on $\mathfrak{n}$ by
\[
X_j f (x) = \frac{d}{ds} \bigg|_{s=0} f(x \exp_N(sX_j))
\]
for $f \in C^{\infty} (N)$ and $x \in N$.
The first-order left-invariant differential operator $X_j$ is homogeneous of degree $v_j$. For a multi-index $\alpha \in \mathbb{N}_0^d$, higher-order differential operators are defined by $X^{\alpha} := X_1^{\alpha_1} X_2^{\alpha_2} \cdots X_d^{\alpha_d}$.
The algebra of all left-invariant differential operators on $N$ is denoted by $\mathcal{D} (N)$.

A function $P : N \to \mathbb{C}$ is a \emph{polynomial} if $P \circ \exp_N$ is a polynomial on $\mathfrak{n}$. Denoting by $\xi_1, ..., \xi_d$ a dual basis of $X_1, ..., X_d$,
the system $\eta_j = \xi_j \circ \exp_N^{-1}$, $j = 1, ..., d$, forms a global coordinate system on $N$.
Each $\eta_j : N \to \mathbb{C}$ forms a polynomial on $N$, and any polynomial $P$ on $N$ can be written uniquely as
\begin{align} \label{eq:polynomial}
P = \sum_{\alpha \in \mathbb{N}_0^d} c_{\alpha} \eta^{\alpha},
\end{align}
where all but finitely many $c_{\alpha} \in \mathbb{C}$ vanish and $\eta^{\alpha} := \eta_1^{\alpha_1} \eta_2^{\alpha_2} \cdots \eta_d^{\alpha_d}$ for a multi-index $\alpha \in \mathbb{N}_0^d$.
The homogeneous degree of $\alpha \in \mathbb{N}_0^d$ is defined as
$
[\alpha] := v_1 \alpha + \cdots + v_d \alpha_d
$
and the homogeneous degree of a polynomial $P$ written as \eqref{eq:polynomial} is
$
d(P) := \max \{ [\alpha] : \alpha \in \mathbb{N}_0^d \; \text{with} \; c_{\alpha} \neq 0 \}.
$

For any $k \geq 0$, the set of polynomials $P$ on $N$ such that $d(P) \leq k$ is denoted by $\mathcal{P}_k$.

\subsection{Schwartz space}
A function $f : N \to \mathbb{C}$ belongs to the Schwartz space $\mathcal{S}(N)$ if $f \circ \exp_N$ is a Schwartz function on $\mathfrak{n}$. A family of semi-norms on $\mathcal{S}(N)$ is given by
\[
\| f \|_{\mathcal{S}, K} = \sup_{|\alpha| \leq K, x \in N} (1+|x|)^K |X^{\alpha} f(x)|, \quad K \in \mathbb{N}_0.
\]
For simplicity, the parameter $K$ is sometimes suppressed from the notation $\| \cdot \|_{\mathcal{S}, K}$ and it is simply written $\| \cdot \|_{\mathcal{S}}$.
The closed subspace of $\mathcal{S}(N)$ of functions with all moments vanishing is defined by
\[
\mathcal{S}_0 (N) = \bigg\{ f \in \mathcal{S}(N) \; : \; \int_N x^{\alpha} f(x) \; d\mu_N (x) = 0, \quad \forall \alpha \in \mathbb{N}_0^d \bigg\}.
\]
For arbitrary $f \in \mathcal{S}(N)$, it will be written $\check{f} (x) := \overline{f(x^{-1})}$ and $f_t(x) := t^{-Q} f(t^{-1} x)$ for $t > 0$.

The dual space $\mathcal{S}'(N)$ of $\mathcal{S}(N)$ is the space of tempered distributions on $N$. If $f \in \mathcal{S}'(N)$ and $\varphi \in \mathcal{S}(N)$, the conjugate-linear evaluation is denoted by $\langle f, \varphi \rangle$. If well-defined, the evaluation is also written as $\langle f, \varphi \rangle = \int_N f (x) \overline{\varphi(x)} \; d\mu_N (x)$ and extends the $L^2$-inner product.
Convolution is defined by
$
f \ast \varphi (x) := \langle f, \check{\varphi}(x^{-1} \cdot) \rangle$ and
$ \varphi \ast f (x) := \langle f, \check{\varphi}(\cdot x^{-1}) \rangle$ for $x \in N$.

\section{Matrix coefficients of quasi-regular representations} \label{sec:matrixcoefficient}
This section is devoted to point-wise estimates and integability properties of the matrix coefficients of a quasi-regular representation.

\subsection{Quasi-regular representation}
Let $N$ be a homogeneous Lie group and let $A = \mathbb{R}^+$ be the multiplicative group. Then $A$ acts on $N$ via automorphic dilations $A \ni t \mapsto \delta_t \in \Aut(N)$.
The semi-direct product $G = N \rtimes A$ is defined with via the operations
\[
(x, t) (y, u) = (x \delta_t (y), t u), \quad (x, t)^{-1} = (\delta_{t^{-1}} (x^{-1}), t^{-1}).
\]
Identity element in $G$ is $e_G = (e_N, 1)$.
The group $G$ is an exponential Lie group, that is, the exponential map $\exp_G : \mathfrak{g} \to G$ is a diffeomorphism, see, e.g. \cite[Proposition 5.27]{fuehr2005abstract}.

The quasi-regular representation $\pi = \ind^G_A (1)$ of $G$ acts unitarily on $L^2 (N)$ by
\[
\pi(x, t) f = t^{- Q/2} f (t^{-1} (x^{-1} \cdot)), \quad (x, t) \in N \times A,
\]
for $f \in L^2 (N)$. Note that $\pi (x,t) = L_x D_t$, where $L_x f = f(x^{-1} \cdot)$ and $D_t f = t^{-Q/2} f(t^{-1} (\cdot))$.

A detailed account on the representation theory of quasi-regular representations of exponential groups can be found in  \cite{lipsman1989harmonic, currey2007admissibility, oussa2013admissibility}, but these results will not be used in this paper.

\subsection{Point-wise estimates} For $f_1, f_2 \in L^2 (N)$, denote the associated matrix coefficient by
\[
V_{f_2} f_1 (x,t) = \langle f_1, \pi(x,t) f_2 \rangle, \quad (x,t) \in N \rtimes A.
\]
The following result provides point-wise estimates for a class of matrix coefficients.

\begin{proposition} \label{prop:pointwise}
Let $f_1, f_2 \in \mathcal{S}_0 (N)$ and $K, M \in \mathbb{N}$ be arbitrary.
\begin{enumerate}[(i)]
\item For all $(x, t) \in N \rtimes A$ with $t \leq 1$,
\begin{align} \label{eq:smallT}
|V_{f_2} f_1 (x,t) | \lesssim  t^{Q/2+M} (1+|x|)^{-K} \| f_1 \|_{\mathcal{S}}  \| f_2 \|_{\mathcal{S}}.
\end{align}
\item For all $(x,t) \in N \rtimes A$ with $t \geq 1$,
\begin{align} \label{eq:largeT}
|V_{f_2} f_1 (x,t) | \lesssim  t^{-(Q/2 + M)} (1 +  |x|)^{-K} \|f _1 \|_{\mathcal{S}} \|f_2\|_{\mathcal{S}}.\end{align}
\end{enumerate}
The implicit constants in \eqref{eq:smallT} and \eqref{eq:largeT} are group constants that depend further only on $M, K$.
\end{proposition}

\begin{proof}
Throughout the proof, a Schwartz semi-norm $\| \cdot \|_{\mathcal{S}, N}$ is simply denoted by $\| \cdot\|_N$.

Let $K, M \in \mathbb{N}$ and let $P = P_{x,M} \in \mathcal{P}_M$ denote the Taylor polynomial of $f \in \mathcal{S} (N)$ at $x \in N$ of homogeneous degree $M$.
By
Taylor's inequality \cite[Theorem 3.1.51]{fischer2016quantization}, there exist constants $c, C > 0$ such that for all
$x,y \in N$,
\[
|f(xy) - P(y) | \leq C \sum_{\substack{|\alpha| \leq M' + 1 \\
[\alpha] > M}} |y|^{[\alpha]} \sup_{|z| \leq c^{M' + 1} |y|} |(X^{\alpha} f)(xz)|,
\]
where $M' := \max\{ | \alpha|  :  \alpha \in \mathbb{N}_0^d \; \text{with} \; [\alpha] \leq M \}$.
For $|\alpha| \leq M' + 1$ and $x, y \in N$,
\begin{align*}
\sup_{|z| \leq c^{M' + 1} |y|} |(X^{\alpha} f)(xz)|
&\leq \| f \|_{K + M' + 1} \sup_{|z| \leq  c^{M' + 1} |y|} (1+|xz|)^{-K} \\
&\lesssim \| f \|_{K + M' + 1} \sup_{|z| \leq  c^{M' + 1} |y|} (1+|x|)^{-K} (1+|z|)^{K} \\
&\lesssim \| f \|_{K + M' + 1}  (1+|x|)^{-K} (1+|y|)^{K}, \\
\end{align*}
where the second line follows from the Peetre-type inequality \cite[Lemma 1.10]{folland1982hardy}. Thus,
\begin{align} \label{eq:global_taylor}
|f(xy) - P(y) | \lesssim \| f \|_{K + M' + 1}  (1+|x|)^{-K}  \sum_{\substack{|\alpha| \leq M' + 1 \\
[\alpha] > M}} |y|^{[\alpha]} (1+|y|)^{K}
\end{align}
for all $x,y \in N$.

(i) Let $(x, t) \in N \rtimes A$ with $t \leq 1$. Then, using that $f_2 \in \mathcal{S}_0 (N)$,
\begin{align*}
 |V_{f_2} f_1 (x,t) | = \bigg| \int_N f_1(xy) D_t \check{f_2}(y^{-1}) \; d\mu_N (y) \bigg|
&\leq \int_N |f_1 (xy) - P(y) | |D_t \check{f_2} (y^{-1}) |\; d\mu_N (y).
\end{align*}
Applying \eqref{eq:global_taylor} thus gives
\begin{align*}
 |V_{f_2} f_1 (x,t) | &\lesssim \| f_1 \|_{K+M' + 1} (1+|x|)^{-K} t^{-Q/2}
 \sum_{\substack{|\alpha| \leq M' + 1 \\ [\alpha] > M}} \int_N |y|^{[\alpha]} |\check{f_2} (t^{-1} y^{-1})| (1+|y|)^K \; d\mu_N (y) \\
 &= \| f_1 \|_{K+M' + 1} (1+|x|)^{-K} t^{Q/2}
 \sum_{\substack{|\alpha| \leq M' + 1 \\ [\alpha] > M}} \int_N |ty|^{[\alpha]} |\check{f_2} ( y^{-1}) |(1+|ty|)^K \; d\mu_N (y) \\
 &\lesssim \| f_1 \|_{K+M' + 1} (1+|x|)^{-K} t^{Q/2+M}
 \int_N |f_2(y)| (1+|y|)^{K+Q(M'+1)} \; d\mu_N (y), \numberthis \label{eq:last_integral}
\end{align*}
where the last inequality uses $[\alpha] \leq Q|\alpha| \leq Q(M' + 1)$.
The integral in \eqref{eq:last_integral} can be estimated by
\begin{align*}
\int_N |f_2(y)| (1+|y|)^{K+Q(M'+1)} \; d\mu_N (y) &\leq \| f_2 \|_{K + Q(M'+1) + Q+1} \int_N (1+|y|)^{-Q-1} \; d\mu_N (y) \\
&\lesssim \| f_2 \|_{K + Q(M'+1) + Q+1}, \numberthis \label{eq:finite_integral}
\end{align*}
where convergence of the integral follows by using polar coordinates  \cite[Proposition 1.15]{folland1982hardy}; see also \cite[Corollary 1.17]{folland1982hardy}.
 A combination of \eqref{eq:finite_integral} and \eqref{eq:last_integral} yields
 the desired claim \eqref{eq:smallT}.

(ii) Note that $|V_{f_2} f_1 (x, t)| = |V_{f_1} f_2 ((x,t)^{-1})|$ for $(x, t) \in N \rtimes A$. Hence, if $t \geq 1$, then it follows by part (i) with $M_0 := M + K$ that
\begin{align*}
|V_{f_2} f_1 (x, t)|  & \lesssim  t^{-(Q/2 + M_0)} (1 + t^{-1} |x|)^{-K} \|f _1 \|_{K+M_0'+1} \|f_2\|_{K + Q(M_0'+1) + Q + 1} \\
&\leq t^{-Q/2 - M} t^{-K} t^K (1+|x|)^{-K} \|f _1 \|_{K+M_0'+1} \|f_2\|_{K + Q(M_0'+1) + Q + 1},
\end{align*}
showing \eqref{eq:largeT}. This completes the proof.
\end{proof}

The estimates provided by Proposition
\ref{prop:pointwise} recover the well-known polynomial localisation for wavelet transforms when $N =\mathbb{R}$, see, e.g. \cite[Section 11-12]{holschneider1995wavelets}.
A similar use of the Taylor inequality  for (compactly supported) atoms can be found
 in \cite[Theorem 2.9]{folland1982hardy}.

\subsection{Analysing vectors} \label{sec:analysing}
Left Haar measure on $G$ is given by $\mu_G (x,t) = t^{-(Q+1)} d\mu_N (x)  dt$ and the modular function is given by $\Delta_G (x,t) = t^{-Q}$. The measure $\mu_G$ is used to define the Lebesgue space $L^p (G) = L^p (G, \mu_G)$ for $p \in [1,\infty]$, and $\|\cdot \|_p$ will denote the $p$-norm.

A measurable function $w : G \to [1, \infty)$ is said to be a \emph{weight} if it is submultiplicative, i.e.,
$w((x,t) (y,u)) \lesssim w(x,t) w(y,u)$ for $(x,t), (y,u) \in G$. A weight $w$ is called \emph{polynomially bounded} if
\begin{align} \label{eq:polynomial_weight}
w(x,t) \lesssim (1+|x|)^k (t^m + t^{-m'}), \quad (x,t) \in G,
\end{align}
for some $k, m, m' \geq 0$. Given such a weight $w$, the weighted Lebesgue space $L^1_w (G)$ consists of all $F \in L^1 (G)$
satisfying $\| F \|_{L^1_w} := \| F w \|_{1} < \infty$.

In \cite{feichtinger1989banach, grochenig1991describing, romero2021dual}, the space of \emph{$w$-analysing vectors} of $\pi$, defined by
\[
\mathcal{A}_w := \bigg\{ g \in L^2 (N) : V_g g \in L^1_w (G) \bigg\},
\]
plays a prominent role.

The following result
provides a simple criterion for analysing vectors:

\begin{lemma} \label{prop:normestimates}
Suppose $g \in \mathcal{S}_0 (N)$. Then
$g \in \mathcal{A}_w$ for any polynomially bounded weight function $w : G \to [1,\infty)$. In particular, the representation $\pi = \ind_A^G (1)$ is integrable.
\end{lemma}
\begin{proof}
Let $k,m,m' \geq 0$ be such that $w(x,t) \lesssim (1+|x|)^k (t^m + t^{-m'})$ for all $(x,t) \in G$.
Then, choosing $K,M,M' \in \mathbb{N}$ sufficiently large, it follows by Proposition \ref{prop:pointwise} that
\begin{align*}
\| V_g g \|_{L^1_w} &\lesssim \int_0^{\infty} \int_N V_{g} g (x,t) (1+|x|)^k (t^m + t^{-m'}) \; d\mu_N (x) \frac{dt}{t^{Q+1}} \\
&\lesssim \int_0^1 t^{Q/2 + M' - m'} t^{-(Q+1)} dt + \int_1^{\infty} t^{-(Q/2 + M)  + m} t^{-(Q+1)} dt < \infty.
\end{align*}
This shows that $g \in \mathcal{A}_w$, and thus $\pi$ is $w$-integrable.
\end{proof}

\section{Admissible vectors} \label{sec:admissible_vectors}
A vector $g \in L^2 (N)$ is said to be \emph{admissible} for the quasi-regular representation $(\pi, L^2 (N))$ if the map
\[
V_{g} : L^2 (N) \to L^{\infty}(G), \quad f \mapsto \langle f, \pi(\cdot) g \rangle
\]
is an isometry into $L^2 (G)$.

\subsection{Reproducing formulae}
The following observation relates admissibility to a Calder\'on-type reproducing formula.

\begin{lemma} \label{lem:admissible_Schwartz}
Let $g \in \mathcal{S} (N)$ with $\int_N g (x) \; d\mu_N (x) = 0$. Then $g$ is admissible if, and only if,
\begin{align} \label{eq:calderon_type}
f = \int_0^{\infty} f \ast \check{g}_t \ast g_t \; \frac{dt}{t} \equiv \lim_{\varepsilon \to 0, \rho \to \infty} \int_{\varepsilon}^{\rho} f \ast \check{g}_t \ast g_t \; \frac{dt}{t}, \quad f \in \mathcal{S} (N),
\end{align}
with convergence in $\mathcal{S}' (N)$.
\end{lemma}
\begin{proof}
 Under the assumptions on $g$, it follows by \cite[Theorem 1.65]{folland1982hardy} that
\[
H_{\varepsilon, \rho} (z) := \int_{\varepsilon}^{\rho} \check{g}_t \ast g_t (z) \; \frac{dt}{t}, \quad z \in N,
 \]
converges in $\mathcal{S}'(N)$ to a distribution $H := \lim_{\varepsilon \to 0, \rho \to \infty} H_{\varepsilon, \rho}$ which is smooth on $N \setminus \{e_N\}$ and homogeneous of degree $-Q$. Let $f \in \mathcal{S} (N)$. Then
 \begin{align*}
\big\| V_g f \big\|_2^2 &=  \lim_{\varepsilon \to 0, \rho \to \infty} \int_{\varepsilon}^{\rho} \int_N | f \ast D_t \check{g} (x) |^2 \; d\mu_G (x,t) \\
 &= \lim_{\varepsilon \to 0, \rho \to \infty} \int_{\varepsilon}^{\rho} \int_N \int_N \int_N
 f(y) \check{g_t} (y^{-1} x) \overline{\check{g_t} (z^{-1} x) f(z)} \;d\mu_N (z) d\mu_N (y) d\mu_N (x) \frac{dt}{t} \\
 &= \lim_{\varepsilon \to 0, \rho \to \infty} \int_{\varepsilon}^{\rho} \int_N \int_N
 f(y) \check{g_t} \ast g_t (y^{-1} z) \overline{f(z)}  \; d\mu_N (y) d\mu_N (z)  \frac{dt}{t} \\
 &= \lim_{\varepsilon \to 0, \rho \to \infty}  \int_N f \ast H_{\varepsilon,\rho} (z) \overline{f(z)} \; d\mu_N (z) \\
 &= \int_N f \ast H(z) \overline{f(z)} \; d\mu_N (z),
 \end{align*}
where the last equality used that $f \ast H_{\varepsilon, \rho} \to f \ast H$ in $\mathcal{S}' (N)$ as $\varepsilon \to 0$ and $\rho \to \infty$.

The map  $f \mapsto f \ast H$ is bounded on $L^2 (N)$ by \cite[Theorem 6.19]{folland1982hardy}. Hence $V_g : \mathcal{S} (N) \to L^2 (G)$ is well-defined,
and it follows that
\begin{align} \label{eq:admissible_schwartz}
 \int_G | \langle f, \pi(x,t) g \rangle |^2 \; d\mu_G (x,t) = \langle f \ast H, f \rangle, \quad f \in L^2 (N).
 \end{align}
Thus $g$ is admissible if, and only if, $\langle f \ast H, f \rangle = \langle f, f \rangle$ for all $f \in L^2 (N)$. Polarisation yields that this is equivalent to \eqref{eq:calderon_type}, which completes the proof.
\end{proof}

The calculations in the proof of Lemma \ref{lem:admissible_Schwartz} are classical, see, e.g. \cite[Theorem 7.7]{folland1982hardy}.

\subsection{Rockland operators} This section provides background on spectral multipliers for Rockland operators, see, e.g. \cite[Chapter 4]{fischer2016quantization} for a detailed account. The stated results will be used in Section \ref{sec:admissible_schwartz} below for the construction of
admissible vectors.

Let $\mathcal{L} \in \mathcal{D}(N)$ be positive and formally self-adjoint.
Then $\mathcal{L}$ is essentially self-adjoint on $L^2 (N)$, and $\mathcal{L}$ will also denote its self-adjoint extension. Let $E_{\mathcal{L}}$ be the spectral measure of $\mathcal{L}$.
For $m \in L^{\infty} (\mathbb{R}_0^+)$, the operator
\[
m(\mathcal{L}) := \int_{\mathbb{R}_0^+} m(\lambda) \; dE_{\mathcal{L}}(\lambda)
\]
is a left-invariant bounded linear operator on $L^2 (N)$. By the Schwartz kernel theorem,
the action of $m(\mathcal{L})$ on $\mathcal{S} (N)$ is given by
\[
m(\mathcal{L}) f = f \ast K_{m (\mathcal{L})}, \quad f \in \mathcal{S} (N),
\]
where $K_{m(\mathcal{L})} \in \mathcal{S}'(N)$ is the associated convolution kernel.

A \emph{Rockland operator} is a homogeneous differential operator $\mathcal{L} \in \mathcal{D} (N)$ of positive degree that is hypoelliptic, i.e. for every distribution $f \in (C_c^{\infty} (N))'$ and every open set $U \subseteq N$,  the condition $(\mathcal{L} f)|_U \in C^{\infty} (U)$ implies that $f|_U \in C^{\infty} (U)$.  Positive Rockland operators are well-known to exist on any graded Lie group.

The following theorem is the key result  used to construct admissible Schwartz functions.

\begin{theorem}[Hulanicki \cite{hulanicki1984functional}] \label{thm:hulanicki}
Let $N$ be a graded Lie group.
Let $\mathcal{L} \in \mathcal{D}(N)$ be a positive Rockland operator and
let $|\cdot| : N \to [0, \infty)$ be a fixed homogeneous quasi-norm on $N$.

For any $M_1 \in \mathbb{N}$, $M_2 \geq 0$, there exist $C = C(M_1, M_2) > 0$ and $k = k(M_1, M_2), k' = k'(M_1, M_2) \in \mathbb{N}_0$ such that, for any $m \in C^k (\mathbb{R}_0^+)$,
the kernel $K_{m(\mathcal{L})}$ of $m(\mathcal{L})$ satisfies
\[
\sum_{[\alpha] \leq M_1} \int_G | X^{\alpha} K_{m(\mathcal{L})} (x)| (1+|x|)^{M_2} \; d\mu_N (x) \leq C
\sup_{\substack {\lambda > 0 \\
\ell = 0, ..., k \\
\ell' = 0, ..., k'}} (1+\lambda)^{\ell'} | \partial_{\lambda}^{\ell} m (\lambda)|.
\]
\end{theorem}

\begin{corollary} \label{cor:hulanicki}
Let $\mathcal{L} \in \mathcal{D}(N)$ be a positive Rockland operator.
\begin{enumerate}[(i)]
\item If $m \in \mathcal{S}(\mathbb{R}_0^+)$, then $K_{m(\mathcal{L})} \in \mathcal{S}(N)$.
\item If $m \in \mathcal{S}(\mathbb{R}_0^+)$ vanishes near the origin, then $K_{m(\mathcal{L})} \in \mathcal{S}_0(N)$.
\end{enumerate}
\end{corollary}

\subsection{Existence of admissible vectors} \label{sec:admissible_schwartz}
The following result yields a class of Schwartz vectors that are admissible.

\begin{proposition} \label{prop:admissible_existence}
Let $N$ be a graded Lie group and let $\mathcal{L} \in \mathcal{D} (N)$ be a positive Rockland operator of degree $\nu$. Let $K_{m (\mathcal{L})}$ be the convolution kernel of a multiplier $m \in \mathcal{S} (\mathbb{R}_0^+)$ satisfying
\begin{align} \label{eq:admissible_r1}
\int_0^{\infty} | m (t)|^2 \; \frac{dt}{t} = \nu.
\end{align}
Then $g := K_{m (\mathcal{L})} \in \mathcal{S} (N)$ is an admissible vector for
 $\pi = \ind^{N \rtimes A}_A (1)$.
\end{proposition}
\begin{proof}
Let $m \in \mathcal{S} (\mathbb{R}_0^+)$ be as in the statement, so that
\begin{align} \label{eq:admissible_r2}
\int_0^{\infty} |m(\lambda t^{\nu} )|^2 \; \frac{dt}{t} = \frac{1}{\nu} \int_0^{\infty} | m (t)|^2 \; \frac{dt}{t} = 1, \quad \text{for all} \;\; \lambda > 0.
\end{align}
By Corollary \ref{cor:hulanicki}, $g := K_{m(\mathcal{L})} \in \mathcal{S} (N)$,
and it suffices to show the reproducing formula \eqref{eq:calderon_type}.
Define $H_{\varepsilon, \rho} := \int_{\varepsilon}^{\rho} \check{g_t} \ast g_t \; t^{-1} dt$ for $0<\varepsilon < \rho < \infty$. Let $f_1, f_2 \in \mathcal{S}(N)$. Then
\begin{align} \label{eq:f1Kf2}
\langle f_1 \ast H_{\varepsilon, \rho} , f_2 \rangle &= \int_{\varepsilon}^{\rho} \langle f_1 \ast \check{g_t} \ast g_t, f_2 \rangle \; \frac{dt}{t} = \int_{\varepsilon}^{\rho} \langle f_1 \ast (\check{g} \ast g )_t, f_2 \rangle \; \frac{dt}{t}.
\end{align}
The spectral theorem implies that $\check{g} \ast g = K_{\overline{m} (\mathcal{L})} \ast  K_{m (\mathcal{L})} = K_{|m|^2 (\mathcal{L})}$. In addition, the homogeneity of $\mathcal{L}$ yields that
$
(\check{g} \ast g)_t = K_{|m|^2 (t^{\nu} \mathcal{L})}
$
for all $t > 0$, see, e.g. \cite[Corollary 4.1.16]{fischer2016quantization}. Combining this with \eqref{eq:f1Kf2} gives
\begin{align*}
\langle f_1 \ast H_{\varepsilon, \rho} , f_2 \rangle
& = \int_{\varepsilon}^{\rho} \big\langle |m|^2 (t^{\nu} \mathcal{L}) f_1 , f_2 \big\rangle \; \frac{dt}{t}
= \int_{\varepsilon}^{\rho} \int_0^{\infty}  |m( t^{\nu} \lambda) |^2 \;  d \langle E_{\mathcal{L}} (\lambda) f_1, f_2 \rangle \frac{dt}{t} \\
&=  \int_0^{\infty}  \int_{\varepsilon}^{\rho} |m( t^{\nu} \lambda) |^2 \;  \frac{dt}{t}  d \langle E_{\mathcal{L}} (\lambda) f_1, f_2 \rangle .
\end{align*}
Hence, by the identity \eqref{eq:admissible_r2},
\[
\lim_{\varepsilon \to 0, \rho \to \infty} \langle f_1 \ast H_{\varepsilon, \rho} , f_2 \rangle
= \int_0^{\infty}  \int_{0}^{\infty} |m( t^{\nu} \lambda) |^2 \;  \frac{dt}{t}  d \langle E_{\mathcal{L}} (\lambda) f_1, f_2 \rangle = \langle f_1, f_2 \rangle.
\]
An application of Lemma \ref{lem:admissible_Schwartz} therefore yields that $g$ is admissible.
\end{proof}

Spectral multipliers for sub-Laplacians on stratified groups were used for constructing admissible vectors in \cite{geller2006continuous}. See also \cite{furioli2006littlewood} for similar discrete Littlewood-Paley decompositions.

\begin{remark} \label{rem:nonhomogeneous}
The use of a \emph{homogeneous} operator is essential in the proof of Proposition \ref{prop:admissible_existence}  to guarantee that the spectral dilates $m(t \cdot ) $, $t > 0$, of a multiplier $m \in \mathcal{S} (\mathbb{R}^+_0)$ yield a convolution kernel $K_{m(t \mathcal{L})}$ that is compatible with automorphic dilations $\{\delta_t \}_{t >0}$.
For non-homogeneous operators, other techniques seem required, see, e.g.   \cite{calzi2020functional, nagel1990harmonic}.
\end{remark}

\subsection{Proof of Theorem \ref{thm:intro}}
Theorem \ref{thm:intro} follows from combining Lemma \ref{prop:normestimates}, Corollary \ref{cor:hulanicki} and Proposition \ref{prop:admissible_existence}.

\section*{Acknowledgment}
J.v.V. gratefully acknowledges support from the Research
Foundation - Flanders (FWO) Odysseus 1 grant G.0H94.18N and the Austrian Science Fund (FWF) project J-4445. Thanks are due to  T. Bruno, M. Calzi and D. Rottensteiner for helpful discussions.

\end{document}